\documentclass{amsart}
\usepackage{amssymb,amsmath,latexsym}
\usepackage{amsthm}
\usepackage{fontenc}
\usepackage{amssymb}
\numberwithin{equation}{section}
\newtheorem{theorem}{Theorem}[section]
\newtheorem{corollary}{Corollary}[theorem]

 \DeclareMathOperator{\erf}{erf}
\DeclareMathOperator*{\Ei}{Ei}

\begin{document}
\author{Alexander E. Patkowski}
\title{\bf On certain sums over primes and the Riesz function}

\maketitle
\begin{abstract}{We offer some comments on series involving the M$\ddot{o}$bius function which approximate sums over primes. To accomplish this, we utilize the derivative of the Gram series by applying Riemann-Stieltjes integration. We offer a new formula connecting the derivative of the Gram series to the Riesz function, which we then use to obtain general integral relationships.}\end{abstract}

\begin{section}{Introduction} 
Let as usual $\pi(x)$ denote the number of primes less than or equal to $x.$ Throughout we will use the notation $f_1(x) \sim f_2(x)$ to mean that $\lim_{x\to \infty}f_1(x)/f_2(x)=1.$ The function $\pi(x)$ has a rich history in the theory of numbers, with approximations (which we denote by $\approx$) offered by Gauss and Legendre [8, pg.41]. Namely, Gauss had offered 
$$\pi(x)\approx \int_{2^{-}}^{x}\frac{dy}{\log(y)}.$$ Another now well-known approximation is the Gram series [8, pg.51, eq.(2.27)]
$$\pi(x)\approx 1+\sum_{n\ge1}\frac{(\log(x))^n}{n!n\zeta(n+1)}:=H(x),$$
where as usual, $\zeta(s)=\sum_{n\ge1}n^{-s}$ is the Riemann zeta function [2, 10] for $\Re(s)>1.$ The extent to which $H(x)$ is a better approximation to $\pi(x)$ than Gauss's logarithmic integral has been detailed in [2, pg.35]. Namely from [2, pg.36], we know that
$$\pi(x)-H(x)=O\left(\frac{\sqrt{x}}{\log(x)}\right),$$
as $x\rightarrow\infty.$
According to a statement by Ramanujan [1, pg. 117, Entry 8],
\begin{equation} \frac{\partial \pi}{\partial x} \approx \frac{1}{x\log(x)}\sum_{n\ge1}\frac{\mu(n)}{n}x^{1/n}.\end{equation} Here $\mu(n)$ is the M$\ddot{o}$bius function [10, pg.3].
It is stated in [1, pg.117] that (1.1) is taken as a formal statement. It appears there is little discussion in the literature on applications of (1.1), or its effectiveness in approximating $\frac{\partial \pi}{\partial x}.$ However, Ramanujan provided an example for the prime sum $\sum_{p}\log(p)p^{-s}$ in [1, pg.117, Entry 7]. Note the right side of (1.1) is monotonically decreasing, and is therefore Riemann integrable. Hence, if we denote the right side as $H'(x),$ then $$\int f(x)dH=\int f(x)H'(x)dx.$$ By Riemann-Stieltjes integration and (1.1),
$$\begin{aligned}\int_{1}^{X} f(x)d\pi&\approx\int_{1}^{X} f(x)dH \\
&=\int_{1}^{X} f(x)H'(x)dx\\
&=-\int_{0}^{-\log(X)}e^{-x}f(e^{-x})H'(e^{-x})dx\\
&=\int_{0}^{-\log(X)}\frac{f(e^{-x})}{x}\left(\sum_{n\ge1}\frac{\mu(n)}{n}e^{-x/n}\right)dx.\end{aligned}$$
If we choose $f(x)=x\log(x),$ then
\begin{equation} \int_{1}^{X} f(x)d\pi=\sum_{p\le X}p\log(p),\end{equation} where the sum on the right is over primes $p.$
On the other hand
$$\int_{0}^{-\log(X)}\frac{f(e^{-x})}{x}\left(\sum_{n\ge1}\frac{\mu(n)}{n}e^{-x/n}\right)dx=-\int_{0}^{-\log(X)}e^{-x}\left(\sum_{n\ge1}\frac{\mu(n)}{n}e^{-x/n}\right)dx.$$
Integrating gives
\begin{equation}\sum_{p\le X}p\log(p)\approx\sum_{n\ge1}\frac{\mu(n)}{n+1}(X^{1+1/n}-1).\end{equation}
The right side of (1.3) converges by a comparison test with the right side of (1.1) and $\sum_{n\ge1}\mu(n)/n=0.$ \par For another interesting application of $H(x),$ first recall [10, pg.52, eq.(3.7.2)], for $\Re(s)>1,$
\begin{equation}\frac{\log(\zeta(s))}{s}-\omega(s)=\int_{2}^{\infty}\frac{\pi(x)}{x^{s+1}}dx,\end{equation}
where $$\omega(s)=\int_{2}^{\infty}\frac{\pi(x)}{x^{s+1}(x^{s}-1)}dx.$$ This formula was used by Titchmarsh to prove the prime number theorem [10, pg.51, eq.(3.7.1)].
 If we employ the approximation $H(x)$ to the right side of (1.4), we get by absolute convergence for $\Re(s)>1,$
$$\frac{\log(\zeta(s))}{s}-\omega(s)= \frac{1}{s2^{s}}+\sum_{n\ge1}\frac{s^{-n-1}\Gamma(n+1,s\log(2))}{nn!\zeta(n+1)}+O\left(\Ei(-(s-\frac{1}{2})\log(2))\right),$$
where we used well-known notation for the incomplete gamma function $$\Gamma(n+1,x)=x^{n+1}\int_{1}^{\infty}y^ne^{-xy}dy,$$ and the exponential integral is given by
$$\Ei(-x):=-\int_{x}^{\infty}\frac{e^{-y}}{y}dy.$$ This estimate involving $\Ei(-x)$ is derived from the estimate of $\pi(x)-H(x)$ given in the introduction [2, pg.36], and the elementary integral evaluation $x>1,$ $\Re(s)>\frac{1}{2},$ $$\int_{y}^{\infty}x^{-s-\frac{1}{2}}ds=\frac{x^{-s-\frac{1}{2}}}{\log(x)}.$$ Since $\Ei(-x)$ decays exponentially for large $x,$ we can write
\begin{equation}\frac{\log(\zeta(s))}{s}-\omega(s)\approx \frac{1}{s2^{s}}+\sum_{n\ge1}\frac{s^{-n-1}\Gamma(n+1,s\log(2))}{nn!\zeta(n+1)}.\end{equation}

Using the known property (the proof involves a change of variables and applies the binomial theorem) that $\Gamma(n+1,x)=n!e^{-x}\sum_{0\le k\le n}x^k/k!,$ we can write the series on the right hand side of (1.5) as 
\begin{equation}\sum_{n\ge1}\frac{s^{-n-1}\Gamma(n+1,s\log(2))}{nn!\zeta(n+1)}=2^{-s}\sum_{n\ge1}\frac{s^{-n-1}}{n\zeta(n+1)}\sum_{0\le k \le n}\frac{(s\log(2))^k}{k!}.\end{equation}
Since $\zeta(n+1)\ge1,$ the series on the right side of (1.6) will converge faster that the geometric series for $\log(1-s^{-1}\log(2))$ when $k=0.$ When $k=n,$ the series is equal to $(H(2)-1)s^{-1}2^{-s}.$ Hence, we have established convergence, as well as a new connection between $H(2)$ and $\log(\zeta(s)).$ To prove the following theorem, we will use the concept of the Mellin transform, which is defined by [6, pg.83, eq.(3.1.11)]
$$M(g)(s):=\int_{0}^{\infty}y^{s-1}g(y)dy.$$ 
The Mellin inversion formula is the line integral [6, pg.89]
$$M^{-1}(g)(y):=\frac{1}{2\pi i}\int_{(c)}y^{-s}M(g)(s)ds.$$ Riemann's function ([8, pg.46, eq.(2.13)] and [8, pg.47, eq.(2.18)], or [2, pg.22]) $J(x)$ is defined for $c>1,$ $x>1,$
$$J(x):=\frac{1}{2\pi i}\int_{(c)}\frac{\log(\zeta(s))}{s}x^sds=\frac{1}{2}\left(\sum_{p^n<x}\frac{1}{n}+\sum_{p^n\le x}\frac{1}{n} \right).$$ 
\begin{theorem} We have for $x>2,$ $c>1,$
$$\begin{aligned}J(x)-\frac{1}{2\pi i}\int_{(c)}\omega(s)x^sds&=1+\sum_{n\ge1}\frac{1}{n\zeta(n+1)}\sum_{0\le k \le n}\frac{(\log(2))^k(\log(x/2))^{n-k}}{k!(n-k)!}\\
&+O\left(\frac{\sqrt{x}}{\log(x)}\right)\\
&\approx 1+\sum_{n\ge1}\frac{1}{n\zeta(n+1)}\sum_{0\le k \le n}\frac{(\log(2))^k(\log(x/2))^{n-k}}{k!(n-k)!}.\end{aligned}
$$ \end{theorem}
\begin{proof} It is known [10, pg.52] that $\omega(s)$ is uniformly convergent for $\Re(s)\ge\frac{1}{2}+\delta,$ $\delta>0,$ and also regular and bounded in this region. We apply Mellin inversion [6, pg.80, eq.(3.1.5)] to (1.5) with $c>1,$ and $x>2.$ For $\Re(s)=c>1,$ we then have for $x>2,$
$$J(x)-\frac{1}{2\pi i}\int_{(c)}\omega(s)x^sds\approx \frac{1}{2\pi i}\int_{(c)}\left(\frac{1}{s2^{s}}+\sum_{n\ge1}\frac{s^{-n-1}\Gamma(n+1,s\log(2))}{nn!\zeta(n+1)}\right)x^sds.$$ We have temporarily dropped the error term to address the main series in the theorem, leaving it for the last step of the proof. Using (1.6), the right hand side becomes,
$$\frac{1}{2\pi i}\int_{(c)}\left(\frac{1}{s2^{s}}+2^{-s}\sum_{n\ge1}\frac{s^{-n-1}}{n\zeta(n+1)}\sum_{0\le k \le n}\frac{(s\log(2))^k}{k!}\right)x^sds.$$ 
Recall [6, pg.95, eq.(3.3.27)] (or differentiating [2, pg.27, eq.(2)]), for $n\ge0,$ $x>1,$ $c>0,$
$$\frac{1}{2\pi i}\int_{(c)}\frac{x^s}{s^{n+1}}ds=\frac{(\log(x))^n}{n!}.$$ Applying this formula gives the right side as
$$1+\sum_{n\ge1}\frac{1}{n\zeta(n+1)}\sum_{0\le k \le n}\frac{(\log(2))^k(\log(x/2))^{n-k}}{k!(n-k)!}.$$ Covergence of this series is similar to the function $H(x)$ upon inspection of the finite sum when $k=0,$ and $k=n.$ Now since (1.5) may instead have an equality when adding the error term $O(-\Ei(-(s-\frac{1}{2})\log(2)),$ we see that taking the Mellin transform of the error term $\sqrt{x}/\log(x)$ gives the original formula. \end{proof}
We remark that the integral on the right hand side of (1.4) can be evaluated using Abel's summation formula to obtain for $\Re(s)>1,$
$$\frac{1}{s}\sum_{p}\frac{1}{p^s}= \frac{1}{s2^{s}}+\sum_{n\ge1}\frac{s^{-n-1}\Gamma(n+1,s\log(2))}{nn!\zeta(n+1)}+O\left(\Ei(-(s-\frac{1}{2})\log(2))\right).$$ As a consequence we have that 
$$\sum_{p}\frac{1}{p^s}\sim \frac{1}{2^{s}}+\sum_{n\ge1}\frac{s^{-n}\Gamma(n+1,s\log(2))}{nn!\zeta(n+1)},$$ as $s\rightarrow\infty.$
\end{section}

\begin{section}{Connection to the Riesz function} 
The Riesz function [9] may take the form 
\begin{equation}\sum_{n\ge1}\frac{\mu(n)}{n^2}e^{-x/n^2}=\frac{1}{2\pi i}\int_{(c)}\frac{\Gamma(s)}{\zeta(2-2s)}x^{-s}ds, \end{equation}
for $0<c<\frac{1}{2}.$ Riesz had replaced $s$ by $1-s$ and multiplied both sides by $x$ in (2.1). It is also noted therein [9] that $$\sum_{n\ge1}\frac{\mu(n)}{n^2}e^{-x/n^2}=O(x^{-3/4+\epsilon}),$$ for every $\epsilon>0$ is equivalent to the Riemann Hypothesis (that all the non-trivial zeros of $\zeta(s)$ are on the line $\Re(s)=\frac{1}{2}$). Recent work on Riesz-type criteria include [4, 5]. However, it appears no practical use of the Riesz function has been made outside of stating variations of Riesz's original criteria. Furthermore, the lack of progress toward the Riemann Hypothesis from these criteria was noted in [10, pg.382]. The purpose of this section is to show its relevance to $H'(x).$ In particular, we will provide some connections to the right hand side of (1.3) as well as other series involving the M$\ddot{o}$bius function.
\par
Note that standard tables for integrals ([3, pg.14, eq.(1)]) together with absolute converge show that
\begin{equation}\int_{0}^{\infty}\sum_{n\ge1}\frac{\mu(n)}{n}e^{-ax/n}\cos(wx)dx=\sum_{n\ge1}\mu(n)\frac{a}{a^2+(wn)^2}, \end{equation}
and
\begin{equation}a\int_{0}^{\infty}\sum_{n\ge1}\frac{\mu(n)}{n^2}e^{-a^2y/n^2}e^{-w^2y}dy=\sum_{n\ge1}\mu(n)\frac{a}{a^2+(wn)^2}. \end{equation}
Therefore, applying Fourier cosine inversion to (2.2) and (2.3) shows (by [3, pg.15, eq.(11)], and [3, pg.7, eq.(1)]) the following formula.
\begin{theorem} For $x>0,$ $a>0,$
\begin{equation}\frac{\pi}{2}\sum_{n\ge1}\frac{\mu(n)}{n}e^{-ax/n}=a\int_{0}^{\infty}\sum_{n\ge1}\frac{\mu(n)}{n^2}e^{-a^2y/n^2}\left(\frac{\sqrt{\pi}}{2}y^{-1/2}e^{-x^2/(4y)} \right)dy.\end{equation}
\end{theorem}
\begin{proof}
We can prove this using Mellin inversion under the assumption of the RH. Parseval's identity is [6, pg.83, eq.(3.1.11)]
\begin{equation}\int_{0}^{\infty}k(y)g(y)dy=\frac{1}{2\pi i}\int_{(c)}M(k)(s)M(g)(1-s)ds. \end{equation}
It is easy to see from changing variables that for $\Re(s)<\frac{1}{2},$
\begin{equation}\int_{0}^{\infty}y^{s-1}\left(y^{-1/2}e^{-x^2/(4y)} \right)dy=\left(\frac{x^2}{2^2}\right)^{s-\frac{1}{2}}\Gamma(\frac{1}{2}-s).\end{equation}
First note [9] that the Riemann Hypothesis is equivalent to the convergence of $1/\zeta(s)$ for $\Re(s)>\frac{1}{2}.$ If we apply (2.1) under the Riemann Hypothesis and (2.6) to Parseval's theorem (2.5), we have the right of (2.4) (excluding the constants) is equal to
$$\frac{1}{2\pi i}\int_{(c)}\left(\frac{x^2}{2^2}\right)^{\frac{1}{2}-s}\Gamma(s-\frac{1}{2})\frac{\Gamma(s)}{\zeta(2-2s)}a^{-2s}ds, $$
for $\frac{1}{2}<c<\frac{3}{4}.$
We shift this integral with $s\rightarrow s+1,$ and use Legrendre's duplication formula $\Gamma(s+1)\Gamma(s+\frac{1}{2})=s\Gamma(s)\Gamma(s+\frac{1}{2})=s\Gamma(2s)2^{1-2s}\sqrt{\pi},$ to get $$\frac{2\sqrt{\pi}}{2\pi i}\int_{(c+1)}x^{-2s-1}\frac{s\Gamma(2s)}{\zeta(-2s)}a^{-2s-2}ds. $$
Replacing $s$ by $s/2$ this is for $-1<d<-\frac{1}{2},$
$$\frac{\sqrt{\pi}}{2\pi i}\int_{(d)}x^{-1-s}\frac{s\Gamma(s)}{\zeta(-s)}a^{-s-2}ds. $$
We move the line of integration to the left by applying Cauchy's residue theorem and computing the residues at the poles $s=-n,$ $n\ge1,$ which gives
$$\begin{aligned}-\sqrt{\pi}\sum_{n\ge2}x^{n-1}a^{n-2}\frac{n(-1)^n}{n!\zeta(n)}&=\sqrt{\pi}\sum_{n\ge1}x^{n}a^{n-1}\frac{(-1)^n}{n!\zeta(n+1)}\\
&=\sqrt{\pi}\sum_{n\ge1}x^{n}a^{n-1}\frac{(-1)^n}{n!}\sum_{k\ge1}\frac{\mu(k)}{k^{n+1}} \\
&=a^{-1}\sqrt{\pi}\sum_{n\ge1}\frac{\mu(n)}{n}(e^{-ax/n}-1)\\
&=a^{-1}\sqrt{\pi}\sum_{n\ge1}\frac{\mu(n)}{n}e^{-ax/n}.\end{aligned}$$
Here the interchange of series is justified by absolute convergence. This completes the second proof of (2.4) after multiplying through by $a.$ 
\end{proof}
Using (2.4) and computations from the introduction we may now state a general result. Let the truncated Laplace transform be denoted by
$$L(X,a):=\int_{0}^{\log(X)}f(x)e^{-ax}dx.$$ 
\begin{corollary} Assuming the integrals converge, we have for $X>0,$

$$\frac{1}{2}\sum_{n\ge1}\frac{\mu(n)}{n}L\left(X,\frac{a}{n}\right)= \int_{0}^{\log(X)}f(x)\int_{0}^{\infty}\sum_{n\ge1}\frac{\mu(n)}{n^2}e^{-y/n^2}\left(\frac{\sqrt{\pi}}{2}y^{-1/2}e^{-x^2/(4y)} \right)dydx.$$
\end{corollary}
\begin{proof} This follows from integrating (2.4) over a suitable function $f(x).$\end{proof}
For our next result, we need a theorem from [7] which involves (1.1). Let $\rho$ denote the non-trivial zeros of $\zeta(s).$
\begin{theorem} ([7, Theorem 2.1]) Let $x>1,$ and assume that the zeros of $\zeta(s)$ are simple. A solution to the singular Fredholm integral equation of the second kind 
\begin{equation} \Delta(x)=f(x)+\int_{0}^{\infty}\sin(xt)\Delta(2\pi t)dt,\end{equation}
\begin{equation}f(x)=\frac{\pi}{2}\sum_{\rho}\frac{x^{-\rho}}{\cos(\pi  \rho/2) \zeta'(\rho)},\end{equation}
is given by \begin{equation}\Delta(x)=\sum_{n\ge1}\frac{\mu(n)}{n}e^{-x/n}.\end{equation}
\end{theorem}
From [3, pg.73], we know that for $\Re(\alpha)>0,$
\begin{equation}\int_{0}^{\infty}e^{-\alpha x^2}\sin(yx)dx=\frac{y}{2\alpha}{}_1F_{1}(1;3/2;-\frac{y^2}{4\alpha}), \end{equation}
where ${}_1F_{1}$ is the confluent hypergeometric function [6, pg.107]
$${}_1F_{1}(a;b;x):=\sum_{n\ge0}\frac{(a)_n}{(b)_n}\frac{x^n}{n!},$$ 
and $(a)_n=a(a+1)\cdots (a+n-1).$
Hence, if we put $a=2\pi$ in (2.4), and take the sine transform, we find by [7, Theorem 2.1] and (2.10) the following using Fubini's theorem.
\begin{theorem} Let $x>1,$ and assume that the zeros of $\zeta(s)$ are simple. Then,
$$\sum_{n\ge1}\frac{\mu(n)}{n}e^{-x/n}-\frac{\pi}{2}\sum_{\rho}\frac{x^{-\rho}}{\cos(\pi  \rho/2) \zeta'(\rho)}$$
$$=2\pi\int_{0}^{\infty}\sum_{n\ge1}\frac{\mu(n)}{n^2}e^{-2\pi y/n^2}\left(\sqrt{\pi}xy^{1/2}{}_1F_{1}(1;3/2;-yx^2) \right)dy.$$
\end{theorem}

We offer a formula for a more general sum than the one noted in the introduction. As usual, $\erf(x)$ is the error function.
\begin{theorem} For $X>1,$ $r\ge1,$
$$\frac{1}{2}\sum_{n\ge1}\frac{\mu(n)}{rn+1}(X^{-(r+1/n)}-1))$$
$$=\int_{0}^{\infty}\sum_{n\ge1}\frac{\mu(n)}{n^2}e^{-y/n^2}\left(e^{yr^2}\left(\erf(2y(r+\log(X)/(2y)))-\erf(r2y) \right) \right)dy.$$
\end{theorem}
\begin{proof}
Note that (by [3, pg.146, eq.(21)] with [3, pg.146, eq.(24), $v=1$]),
\begin{equation}\int_{0}^{\log(X)}e^{-rx-ax^2}dx=\frac{\sqrt{\pi}e^{r^2/(4a)}\left(\erf(\frac{r+2a\log(X)}{2\sqrt{a}})-\erf(\frac{r}{2\sqrt{a}}) \right)}{2\sqrt{a}}. \end{equation}
In Corollary 2.1.1 we put $f(x)=e^{-rx},$ $r\ge1,$ using (2.11) to find that 
$$\frac{1}{2}\sum_{n\ge1}\frac{\mu(n)}{rn+1}(X^{-(r+1/n)}-1)$$
$$\begin{aligned}=& \int_{0}^{\log(X)}e^{-rx}\int_{0}^{\infty}\sum_{n\ge1}\frac{\mu(n)}{n^2}e^{-y/n^2}\left(\frac{\sqrt{\pi}}{2}y^{-1/2}e^{-x^2/(4y)} \right)dydx \\
=&\int_{0}^{\infty}\sum_{n\ge1}\frac{\mu(n)}{n^2}e^{-y/n^2}\left(\frac{\pi}{2}e^{yr^2}\left(\erf(2y(r+\log(X)/(2y)))-\erf(r2y) \right) \right)dy.\end{aligned} $$
The interchange of integrals is justified by Fubini's theorem.\end{proof}
\end{section}

1390 Bumps River Rd. \\*
Centerville, MA
02632 \\*
USA \\*
E-mail: alexpatk@hotmail.com, alexepatkowski@gmail.com

\begin{thebibliography}{9}

\bibitem{ConcereteMath} Berndt, B. C. Ramanujan's Notebooks, Part IV. New York: Springer--Verlag, pp. 134-135, 1994.


\bibitem{ConcreteMath} H. M. Edwards. Riemann's Zeta Function, 1974. Dover Publications.

\bibitem{ConcreteMath} Erde'lyi, A. (ed.) (1954). Tables of Integral Transforms, Vol.1 (McGraw-Hill, New York).

\bibitem{ConcreteMath} Garg, M., Maji, B. \emph{Hardy--Littlewood--Riesz type equivalent criteria for the Generalized Riemann hypothesis.} Monatsh. Math. 201, 771--788 (2023).

\bibitem{ConcreteMath} Gupta, S., Vatwani, A. \emph{Riesz-type criteria for L-functions in the Selberg class,} Canadian Journal of Mathematics, 1--27, (2023).

\bibitem{ConcreteMath} R. B. Paris, D. Kaminski, Asymptotics and Mellin--Barnes Integrals. Cambridge University Press. (2001)

\bibitem{ConcreteMath} A. E. Patkowski, \emph{On some instances of Fox's Integral Equation,} Archivum Mathematicum, Volume 55 (2019), No.3, pp.195--201.

\bibitem{ConcreteMath} H. Riesel, Prime Numbers and Computer Methods for Factorization, Birkh$\ddot{a}$user, Boston, MA,1985.
\bibitem{ConcreteMath} M. Riesz, \emph{Sur l'hypoth'ese de Riemann,} Acta Math., 40 (1916),185--190.

\bibitem{ConcereteMath} E. C. Titchmarsh, The theory of the Riemann zeta function, Oxford University Press, 2nd edition, 1986.

\end{thebibliography}
\end{document}